\documentclass[preprint,12pt]{elsarticle}
\usepackage{amssymb,amsmath}

\usepackage{graphicx}

\usepackage{amssymb}
\usepackage{amsthm}

\newtheorem{theorem}{Theorem}
\newtheorem{corollary}[theorem]{Corollary}
\newtheorem{lemma}[theorem]{Lemma}
\newtheorem{definition}[theorem]{Definition}
\newtheorem{proposition}[theorem]{Proposition}

\begin{document}

\begin{frontmatter}



\title{Paths and animals in unbounded degree  graphs with repulsion}


\author[dk]{Dorota K\c{e}pa-Maksymowicz}
\author[jk]{Yuri Kozitsky}

\address[dk]{dkm@umcs.lublin.pl\\
Instytut Matematyki, Uniwersytet Marii Curie-Sk{\l}odowskiej, 20-031 Lublin, Poland}

\address[jk]{jkozi@hektor.umcs.lublin.pl\\
Instytut Matematyki, Uniwersytet Marii Curie-Sk{\l}odowskiej, 20-031 Lublin, Poland
}

\begin{abstract}

A class of countable infinite graphs with unbounded vertex degree is considered. In these graphs, the vertices of large degree `repel' each other, which means that the path distance between two such vertices cannot be smaller than a certain function of their degrees. Assuming that this function increases sufficiently fast, we prove that the number of finite connected subgraphs (animals) of order $N$ containing a given vertex $x$ is exponentially bounded in $N$ for $N$ belonging to an infinite subset $\mathcal{N}_x\subset \mathbb{N}$. Under a less restrictive condition, the same result is obtained for the number of simple paths originated at a given vertex. These results are then applied to a number of problems, including estimating the growth of the Randi\'c index and of the number of greedy animals.
\end{abstract}


\end{frontmatter}



\section{Introduction}
\label{1S}

 Infinite graphs are used in probabilistic
combinatorics, image processing, and many other domains. In
particular, they serve as underlying discrete metric spaces for
Markov random fields \cite{Burioni,FP,Hag,KK,KKU,KKP,PU,Procacci}.
The structure of such graphs is more accessible for studying if the
vertex degrees are globally bounded. However, in many important
applications it is essential to employ unbounded degree graphs, see
\cite{Hag,KK,KKP}. For these graphs, it is intuitively clear that
their global metric properties can be similar to those of bounded
degree graphs if the vertices with large degree are `sparse', see,
e.g., the Introduction in \cite{Procacci}. In the present work, we
consider two families of `repulsive graphs', in which vertices with large degree `repel' each
other in the sense of Definition \ref{2df} below. For such graphs,
we derive exponential upper bounds for the number of connected
subgraphs of order $N$ which contain a given vertex, valid for large
$N$. These results allow for obtaining similar estimates also for
other metric characteristics, e.g., for the number of vertices in a
ball of radius $N$. In Section \ref{11S}, we introduce necessary
notions and notations and then formulate our main results in
Theorems \ref{1tm}, \ref{2tm} and Corollary \ref{1co}. In Section
\ref{2S}, we describe some applications of these results. Among them
we note an upper estimate for the generalized Randi\'c index
(Proposition \ref{RIpn}) and an almost sure sublinear growth of
weights of greedy graph animals (Proposition \ref{5apn}). In Section
\ref{4S}, we give the proof of the statements just mentioned
preceded by the study of the properties of paths and animals in
repulsive graphs conducted in Section \ref{3S}. Here we introduce
the notion of a tempered graph by imposing restrictions on the
vertex degree growth, see Definition \ref{goodadf}. Then in Lemmas
\ref{gpn2} and \ref{gepn}, we show that the properties stated in
Theorems \ref{1tm} and \ref{2tm} hold for such tempered graphs.
Thereafter, by a technical result obtained in Lemma \ref{anpn} we
prove Lemma \ref{Qpn} which gives us tools for controlling the
vertex degree growth in the repulsive graphs which we study. By
means of these tools we prove that our graphs are tempered, which
yields the proof of Theorems \ref{1tm} and \ref{2tm}, as well as of
Propositions \ref{RIpn} and \ref{5apn}.

\section{Setup and results}
\label{11S}

 Let ${\sf G} = ({\sf V}, {\sf E})$ be a countably
infinite simple graph with no loops. By writing $x \sim y$ we mean
that $x, y \in {\sf V}$ constitute an edge, $\langle x, y \rangle =
\langle y, x \rangle \in {\sf E}$.  We say that such $x$ and $y$ are
adjacent and that they are the endpoints of the edge $\langle x, y
\rangle$. For each $x\in {\sf V}$, the {\it degree} $$n(x)\
\stackrel{\rm def}{=} \  \#\{y\in {\sf V}: y\sim x\}$$ is assumed to
be finite, whereas
\begin{equation}
  \label{N}
  {n}_{\sf G} \ \stackrel{\rm def}{ =} \ \sup_{x\in {\sf V}} n(x),
\end{equation}
can be finite or infinite. A finite connected subgraph, ${\sf A}
\subset {\sf G}$, is called an {\it animal} (also a {\it polymer},
cf. \cite{FP,KP,PU}). By ${\sf V}({\sf A})$ and ${\sf E}({\sf A})$
we denote the set of vertices and edges of ${\sf A}$, respectively.
A {\it path}, $\vartheta$, is a finite sequence of vertices, $\{x_0,
x_1, \dots, x_n\}$, not necessarily distinct, such that $x_{k+1}
\sim x_k$, for all $k=0, \dots , n-1$. Then $\vartheta$ {\it
originates} at $x_0$ and {\it terminates} at $x_n$. Its {\it length}
$|\vartheta|$ is set to be $n$. In a {\it simple path}, all $x_0,
x_1, \dots, x_{n-1}$ are distinct. By ${\sf G}_\vartheta$ we denote
the graph generated by $\vartheta$. That is, its vertex set ${\sf
V}_\vartheta$ consists of those in $\vartheta$, not counting
repeated vertices; the edge set ${\sf E}_\vartheta$ comprises the
edges with both endpoints in  ${\sf V}_\vartheta$. Clearly, each
${\sf G}_\vartheta$ is an animal.

By $\vartheta (x,y)$ we denote a path such that $x_0=x$ and $x_n=y$.
The path distance $\rho(x,y)$ is set to be the length of the
shortest path $\vartheta (x,y)$. A ball ${\sf B}_N(x)$ (resp., a
sphere ${\sf S}_N(x)$), $N\in \mathbb{N}$ and $x\in {\sf V}$, is the
set of $y\in {\sf V}$ such that $\rho(x,y)\leq N$ (resp.,
$\rho(x,y)= N$). For $N\in \mathbb{N}$ and $x\in {\sf V}$, let
$\mathcal{A}_N(x)$ denote the set of all animals such that $x \in
{\sf V}({\sf A})$ and $|{\sf V}({\sf A})|=N$. For such $x$ and $N$,
let also $\varSigma_N(x)$ be the set of all simple paths of length
$N$ originated at $x$. In many applications, see
\cite{Cox,Dob,FP,H,KP,M,PU}, one needs to estimate the growth of the
cardinalities of the mentioned sets as $N\to +\infty$. For a graph ${\sf
G}$ with $n_{\sf G}< \infty$, there exist positive $q_{\sf G}$,
$\bar{q}_{\sf G}$, and ${N}_{\sf G}$ such that the following
estimates hold
\begin{equation}
 \label{U1}
{\rm (a)} \quad \  |\varSigma_N(x) | \leq q_{\sf G}^N, \qquad \quad
{\rm (b)} \quad \ |\mathcal{A}_N (x) | \leq \bar{q}_{\sf G}^N,
\end{equation}
for all $N\geq {N}_{\sf G}$. The first estimate can easily be proven
to hold with $q_{\sf G}= n_{\sf G}$ and $N_{\sf G}=1$, cf.
(\ref{13}) below. The second one is not so immediate, see \cite[Chapter 2]{MM}
where a more general estimate was proved. Note that (b) implies (a).
By (a) in (\ref{U1}) one readily gets
\begin{equation}
  \label{U2}
{\rm (a)} \quad \  |{\sf S}_N(x) | \leq q_{\sf G}^N, \qquad \quad
{\rm (b)} \quad \ |{\sf B}_N (x) | \leq \frac{{q}_{\sf G}}{{q}_{\sf
G}-1}{q}_{\sf G}^N,
\end{equation}
For graphs with $n_{\sf G}= +\infty$, the cardinalities in
(\ref{U1}) can grow faster than exponentially. Furthermore, if (a)
or (b) holds for $N\geq N_*$ with one and the same $N_*$ for all
$x\in {\sf V}$, then $n_{\sf G} < \infty$.

For $x,y\in {\sf V}$, we set
\begin{equation*}
 m_{+} (x,y) = \max\{n(x); n(y)\}, \qquad  m_{-} (x,y) = \min\{n(x);
 n(y)\}.
\end{equation*}
\begin{definition}
  \label{2df}
Let $\phi:\mathbb{N} \to (0,+\infty)$ be strictly increasing. By
$\mathbb{G}_{\pm}(\phi)$ we denote the family of graphs, for each of
which there exists an integer $n_*$ such that
\begin{equation}
  \label{7}
  \rho(x,y) \geq \phi(m_{\pm}(x,y)),
\end{equation}
whenever $m_{-} (x,y) > n_{*}$. No restrictions are imposed if
$m_{-} (x,y) \leq n_{*}$.
\end{definition}
Note that, for the first time, a condition like (\ref{7}) appeared in \cite{BD}.
Clearly,
\begin{equation*}
\mathbb{G}_{+}(\phi)\subset \mathbb{G}_{-}(\phi).
\end{equation*}
For  ${\sf G}\in\mathbb{G}_{\pm}(\phi)$, by  (\ref{7}) vertices of
large degrees `repel' each other. That is why we call these graphs
{\it repulsive}. For such a graph ${\sf G}=({\sf V}, {\sf E})$, we
set
\begin{equation}
  \label{NU}
{\sf V}_* = \{ x\in {\sf V} \ | \ n(x) \leq n_*\}, \qquad {\sf
V}_*^c = {\sf V} \setminus {\sf V}_*,
\end{equation}
and consider
\begin{equation*}
   {\sf K} (x)\
\stackrel{\rm def}{ =} \  \{ y \in {\sf V} \ | \ \rho (y,x) < \phi
\left(n(x)\right)\}.
\end{equation*}
Now let ${\sf G}$ be in $\mathbb{G}_{+}(\phi)$. For  $x\in {\sf
V}^c_*$,   by (\ref{7}) we have ${\sf K} (x){\cap} {\sf V}_*^c =
\{x\}$, i.e., $x$ `repels' all  vertices $y\in{\sf V}_*^c$ from
${\sf K} (x)$. For the sake of convenience, we shall assume that
${\sf K}(x)$ contains the neighborhood of $x$, which is equivalent
to assuming $\phi (4) >1$. Then for any ${\sf G} \in \mathbb{G}_{+}
(\phi)$, we have
\begin{equation*}
\phi (n_*+1) > 1.
\end{equation*}
For ${\sf G}\in\mathbb{G}_{-}(\phi)$, $x$ `repels' from ${\sf K}(x)$
only those $y\in {\sf V}_*^c$, for which $n(y) \geq n(x)$.

Our main results are contained in the following two statements.
\begin{theorem}
\label{1tm} Let $\phi:\mathbb{N} \rightarrow (0, +\infty)$ be such
that the following holds
\begin{equation}
\label{U3} \sum_{k=1}^\infty \frac{\log t_{k+1}}{\phi(t_k)} <
\infty,
\end{equation}
for some strictly increasing sequence $\{t_k\}_{k\in
\mathbb{N}}\subset \mathbb{N}$. Then, for each ${\sf
G}\in\mathbb{G}_{-} (\phi)$, there exists $q_{\sf G}
>1$ such that, for any $x\in {\sf V}$, there exists a
strictly increasing sequence $\{N_k\}_{k\in \mathbb{N}}\subset
\mathbb{N}$ such that the estimate
\begin{equation}
  \label{g5m}
\left\vert \mathit{\Sigma}_{N} (x) \right\vert \leq q_{\sf G}^{N}
\end{equation}
holds for all $N= N_k$, $k\in \mathbb{N}$. If ${\sf
G}\in\mathbb{G}_{+} (\phi)$, then, for each $x\in {\sf V}$, there
exists $N_x \in \mathbb{N}$ such that the estimate (\ref{g5m}) holds
for all $N\geq N_x$.
\end{theorem}
\begin{theorem}
\label{2tm} Let $\phi:\mathbb{N} \rightarrow (0, +\infty)$ be such
that the following holds
\begin{equation}
\label{U4} \sum_{k=1}^\infty \frac{t_{k+1}\log t_{k+1}}{\phi(t_k)} <
\infty,
\end{equation}
for some strictly increasing sequence $\{t_k\}_{k\in
\mathbb{N}}\subset \mathbb{N}$. Then, for each ${\sf
G}\in\mathbb{G}_{-} (\phi)$, there exists $\bar{q}_{\sf G}
>1$ such that, for any $x\in {\sf V}$, there exists a
strictly increasing sequence $\{N_k\}_{k\in \mathbb{N}}\subset
\mathbb{N}$ such that the estimate
\begin{equation}
  \label{g5m2}
\left\vert \mathcal{A}_N (x) \right\vert \leq \bar{q}_{\sf G}^{N}
\end{equation}
holds for all $N= N_k$, $k\in \mathbb{N}$. If ${\sf
G}\in\mathbb{G}_{+} (\phi)$, then, for any $x\in {\sf V}$, there
exists $\overline{N}_x \in \mathbb{N}$ such that the estimate
(\ref{g5m2}) holds for all $N\geq \overline{N}_x$.
\end{theorem}
An immediate corollary of Theorem \ref{1tm} is the following
statement.
\begin{corollary}
  \label{1co}
Let ${\sf G}$ be in $\mathbb{G}_{+} (\phi)$ with $\phi$ obeying
(\ref{U3}). Let also $N_x$, $x\in {\sf V}$, be as in Theorem
\ref{1tm}. Then there exists $B_x>0$ such that, for all $N>N_x$, the
following holds
\begin{equation}
  \label{U5}
{\rm (a)} \quad \  |{\sf S}_N(x) | \leq q_{\sf G}^N, \qquad \quad
{\rm (b)} \quad \ |{\sf B}_N (x) | \leq B_x q_{\sf G}^N.
\end{equation}
\end{corollary}
\begin{proof}
 By the very definition of ${\sf S}_N(x)$, we have that $|{\sf
 S}_N(x)| \leq |\varSigma_N (x)|$, which yields (a) in (\ref{U5}),
whereas (b) with $B_x =|{\sf B}_{N_x} (x)|/q_{\sf G}^{N_x}$ follows
by
 (a).
\end{proof}
The optimal choice of $\{t_k\}_{k\in \mathbb{N}}$ in (\ref{U3}) seems
to be $t_k = \exp(e^k)$, for big enough $k$. Then the choice of
$\phi$ can be $\phi (t) = \upsilon \log t (\log \log
t)^{1+\epsilon}$, $\epsilon>0$; cf. \cite[Theorem 4]{KKP}.

\section{Applications}
\label{2S}
\subsection{Percolation}

Let ${\sf G}= ({\sf V}, {\sf E})$ be as above. For ${\sf E}' \subset
{\sf E}$, we set ${\sf G}' =({\sf V}, {\sf E}')$. Note that ${\sf
G}'$ need not be  connected. Let now edges $e\in {\sf E}'$ be picked
at random, independently and with the same probability $p$ each.
This defines a probability measure, $\mu_p^{\rm b}$, on the set of
all subsets of ${\sf E}$. The corresponding subgraph ${\sf G}'$ with
randomly picked $e\in {\sf E}'$ is random as well. The event that it
has an infinite connected component (called also {\it cluster})
occurs with probability either zero or one, dependent on the value
of $p$. This is the Bernoulli bond percolation model, cf.
\cite{Gr,Hag}.
\begin{proposition}
  \label{1pn}
Let $\phi$ obey (\ref{U3}) and ${\sf G}$ be in $\mathbb{G}_{-}
(\phi)$, so that (\ref{g5m}) holds. Then no cluster appears
$\mu_p^{\rm b}$-almost surely whenever $p < 1/q_{\sf G}$.
\end{proposition}
\begin{proof}
Given $x \in {\sf V}$, the probability that there exists at least
one simple path of length $N$ originated at $x$ does not exceed $p^N
|\varSigma_N(x)|$. Then the proof follows by (\ref{g5m}) and the
Borel-Cantelli lemma.
\end{proof}
Now for ${\sf V}'\subset {\sf V}$, let ${\sf E}'\subset {\sf E}$
comprise the edges with both endpoints in ${\sf V}'$. Set ${\sf G}'=
( {\sf V}', {\sf E}')$. Further, suppose that each vertex of ${\sf
V}'$ is picked at random, independently and with the same
probability $p$ each. This defines a probability measure,
$\mu_p^{\rm s}$, on the set of all subsets of ${\sf V}$. Thereby,
the subgraph ${\sf G}'$ is random. The event that it has a cluster
occurs with probability either zero or one, dependent on $p$. The
appearance of a cluster is called the Bernoulli site percolation,
see \cite[Chapter 3]{Gr} or \cite{Hag}.

\begin{proposition}
  \label{2pn}
Let $\phi$ obey (\ref{U4}) and ${\sf G}$ be in
$\mathbb{G}_{-}(\phi)$, so that (\ref{g5m2}) holds. Then no cluster
appears $\mu^{\rm s}_{p}$-almost surely whenever $p < 1/
\bar{q}_{\sf G}$.
\end{proposition}
\begin{proof}
Given $x\in {\sf V}$, the probability that there exists at least one
connected subgraph of order $N$, which contains $x$, does not exceed
$p^N |\mathcal{A}_N (x)|$. Then the proof follows by (\ref{g5m2})
and the Borel-Cantelli lemma.
\end{proof}
Further applications of the above results to models of dependent
percolation, e.g., to the random cluster model, can be developed by
means of cluster expansion techniques
\cite{Dob,FP,KP,MM,PU,Procacci}. Applications in \cite{KK,KKU,KKP}
to Gibbs random fields on the graphs considered here are based on the estimate in
(\ref{g5m}).

\subsection{Randi$\acute{c}$ index}
For a real $\theta$ and an animal, ${\sf A}$, we set
\begin{equation*}
R^{\theta}({\sf A})=\sum_{\langle x,y\rangle\in {\sf
E({\sf A})}}[n(x)n(y)]^{\theta}.
\end{equation*}
In mathematical chemistry, large molecules are consideerd as finite
trees. For such a tree ${\sf T}$, $R^{\theta}({\sf T})$ is known
under the name {\it generalized Randi$\acute{c}$} or {\it
connectivity} index, see  \cite{MC}. It turns out that its value is
closely related to chemical properties of the corresponding
substance.

For a vertex $x$ and $N\in \mathbb{N}$, we define
\begin{equation*}
R^{\theta}_{N}(x)= \max_{{\sf A}\in{\mathcal A}_N(x)}
R^{\theta}({\sf A}).
\end{equation*}
\begin{proposition}
\label{RIpn} Let $\phi:\mathbb{N} \rightarrow (0, +\infty)$ be
strictly increasing and such that
\begin{equation}\label{warR}
\sum_{k=1}^\infty \frac{t_{k+1}^{\theta+1}}{\phi(t_k)} < \infty,
\end{equation}
for some strictly increasing sequence $\{t_k\}_{k\in
\mathbb{N}}\subset \mathbb{N}$. Then, for each ${\sf G}\in
\mathbb{G}_{-} (\phi)$, there exists $\tilde{q}_{\sf G}>1$ such
that, for any $x\in {\sf V}$, there exists a strictly increasing
sequence $\{N_k\}_{k\in \mathbb{N}}\subset \mathbb{N}$ such that
\begin{equation}\label{radic}
R^{\theta}_{N}(x)\leq \tilde{q}_{\sf G}^N
\end{equation}
holds for all $N= N_k$, $k\in \mathbb{N}$. If ${\sf
G}\in\mathbb{G}_{+} (\phi)$, then for any $x\in {\sf V}$, there
exists $\widetilde{N}_x \in \mathbb{N}$ such that (\ref{radic})
holds for all $N\geq \widetilde{N}_x$.
\end{proposition}
The proof of this statement will be given below.

\subsection{Growth of ${\rm Aut}({\sf G})$}

For a ${\sf G}=({\sf V}, {\sf E})$, an automorphism, $\gamma$, is a
bijection ${\sf V}\ni x \mapsto x \gamma\in {\sf V}$ such that $x
\sim y$ implies $x\gamma \sim y\gamma$. The automorphisms constitute
a group, denoted by ${\rm Aut}({\sf G})$. Assume that ${\sf V}$ is
given the discrete topology, and let $\mathcal{T}$ be the weakest
topology on ${\rm Aut}({\sf G})$ in which the maps ${\rm Aut}({\sf
G})\ni\gamma \mapsto x \gamma \in {\sf V} $ are continuous for all
$x\in {\sf V}$. It is known \cite{Adams} that $({\rm Aut}({\sf G}),
\mathcal{T})$ is a locally compact Polish group. By the local
compactness, there exists a right Haar measure on ${\rm Aut}({\sf
G})$, which we denote by $\mu$. For $x\in {\sf V}$, the set
\[
{\sf \Gamma}_x := \{ \gamma \in {\rm Aut}({\sf G}): x \gamma = x\}
\]
is the {\it stabilizer} of $x$. It is compact and open, and thus $0<
\mu({\sf \Gamma}_x) < \infty$, for all $x\in {\sf V}$. Let ${\sf
\Delta}$ stand for a compact neighborhood of the identity of ${\rm
Aut}({\sf G})$. For $n\in \mathbb{N}$, by ${\sf \Delta}^n$ we denote
the set of all products $\gamma_1 \gamma_2 \cdots \gamma_n$ of the
elements of ${\sf \Delta}$.
\begin{proposition}
  \label{4pn}
Let ${\sf G}$ be in $\mathbb{G}_{+} (\phi)$ with $\phi$ obeying
(\ref{U3}), and let ${\sf \Delta}$ be as above. Then there exist
$C>0$ and $N_* \in \mathbb{N}$ such that, for all $N\geq N_*$, the
following holds
\begin{equation}
  \label{U7}
  \mu({\sf \Delta}^N) \leq C q_{\sf G}^N,
\end{equation}
where $q_{\sf G}$ is the same as in (\ref{g5m}).
\end{proposition}
\begin{proof}
By Proposition 3.2 in \cite{Adams}, for each $x\in {\sf V}$, there
exists $c>0$ such that, for all $N\in \mathbb{N}$, the following
estimate holds
\[
 \mu({\sf \Delta}^N)   \leq \mu({\sf \Gamma}_x) |{\sf B}_{cN}
 (x)|.
\]
We fix $x$ and apply (b) of (\ref{U5}), which yields (\ref{U7}) with
$N_* = N_x$ and $C = B_x \mu({\sf \Gamma}_x)q^c_{\sf G}$.
\end{proof}

\subsection{Greedy animals}

Let $\{Y_x: x\in {\sf V}\}$ be a family of independent positive
random variables (weights). For $N\in \mathbb{N}$ and $x\in {\sf V}$, we define
\begin{equation}
\label{SNx}
S_N(x)=\max_{{\sf A}\in \mathcal{A}_N(x)} \sum_{x\in {\sf V}({\sf A})} Y_x.
\end{equation}
Those ${\sf A}\in \mathcal{A}_N(x)$, for which the maximum in
(\ref{SNx}) is attained are called {\it greedy  animals}, see
\cite{Cox} for motivating examples, applications, and further details.

Let $P_x$ be the probability measure on $[0,+\infty)$ which is the
law of $Y_x$. Then the law of the family $\{Y_x : x\in {\sf
V}\}$ is defined as a product measure in a standard way. We assume
that, for each $x\in {\sf V}$,
\begin{equation}
\label{Ex} w_x (t):= \log \mathbb{E} e^{tY_x } <\infty,
\end{equation}
for a certain $t>0$. Thus, $w_x$ is analytic in some neighborhood of
$t=0$, and hence $w_x(t)/t \to v_x$ as $t\to 0$, where
\begin{equation}
  \label{Ex1}
v_x := \mathbb{E} Y_x.
\end{equation}
\begin{proposition}
  \label{5apn}
Let ${\sf G}$ be in $\mathbb{G}_{-} (\phi)$ with $\phi$ satisfying
(\ref{U4}). Suppose also that
\begin{equation}
  \label{Ex2}
 v_x \leq C n(x) \log n(x),
\end{equation}
for some $C>0$ and each $x\in {\sf V}$. Then there exists $Y>0$ such
that, for each $x\in {\sf V}$, there exists an increasing sequence,
$\{N_k\}_{k\in \mathbb{N}}\subset \mathbb{N}$, for which
\begin{equation}
  \label{ex3}
 \limsup_{k\to +\infty} \frac{1}{N_k}S_{N_k}(x) \leq Y \qquad \ \ {\rm with} \ {\rm probability} \ 1.
\end{equation}
\end{proposition}
The proof of this statement will be given below. Let us now make
some comments. The lattice $\mathbb{Z}^d$ can be turned into a graph
by setting $x\sim y$ if $|x-y|=1$. The greedy  animals on
$\mathbb{Z}^d$ were studied in detail in \cite{Cox,Gand,H,M}. In
those papers, however, the weights are supposed to be identically
distributed with law $P_x$ satisfying less restrictive conditions
(as compared to (\ref{Ex})), involving the lattice dimension $d$,
cf. Theorem 1 in \cite{Cox} or Theorem 3.3 in \cite{M}. In the
statement above, we allow the mean value of $Y_x$ to increase in a
controlled way (\ref{Ex2}), which seems to be quite natural for
unbounded degree graphs. In a separate work, we shall study greedy
animals in such graphs in more detail. In particular, we plan to
relax the exponential integrability assumed in (\ref{Ex}).

\section{Further properties of paths and animals}
\label{3S}

\subsection{Counting paths}

We recall that by ${\sf G}_\vartheta = ({\sf V}_\vartheta, {\sf
E}_\vartheta)$ we denote the graph generated by path  $\vartheta$.
For $e\in {\sf E}$, we say that $\vartheta$ traverses $e$ if $e\in
{\sf E}_\vartheta$. We say that $\vartheta = \{x_0, \dots , x_n\}$
leaves $x_{k}$ towards $x_{k+1}$, $k=0, \dots, n-1$. For $x\in {\sf
V}_\vartheta$, let $\nu_\vartheta(x)$ be the number of times
$\vartheta$ leaves $x$. We also set $\nu_\vartheta(x)=0$ if $x$ is
not in ${\sf V}_\vartheta$. Then, for a simple path, $\nu_\vartheta
(x) \leq 1$. Recall that $\mathit{\Sigma}_N (x)$ denotes the
collection of simple paths of length $N\in \mathbb{N}$ originated at
a given $x\in {\sf V}$. Along with this set we also consider
$\mathit{\Theta}_N (x)$ being the collection of paths
$\vartheta=\{x, x_1, \dots , x_N\}$ such that the number of times
$\vartheta$ leaves each $y\in {\sf V}_\vartheta$ towards any $z\in
{\sf V}_\vartheta$ is at most one. Note that this does not mean
$\nu_\vartheta (x) \leq 1$.
\begin{lemma}
  \label{pathpn}
For any $x\in {\sf V}$ and $N\in \mathbb{N}$, it follows that
$\mathit{\Sigma}_N (x) \subset \mathit{\Theta}_N (x)$. Each
$\vartheta \in \mathit{\Theta}_N (x)$ has the properties: (i) each
$e \in {\sf E}_\vartheta$ can be traversed by $\vartheta$ at most
twice; (ii) $\nu_\vartheta (y) \leq n(y)$ for each $y \in {\sf
V}_\vartheta$.
\end{lemma}
{\it Proof:} The stated inclusion  is immediate, whereas both (i) and (ii)  follow from
the fact that $\vartheta$ leaves each $x\in {\sf V}_\vartheta$
towards any $y\sim x$ at most once.   $\square$

\begin{lemma}
  \label{pathpn1}
For any $x\in {\sf V}$ and $N\in \mathbb{N}$, it follows that
\begin{equation}
  \label{12}
\left\vert \mathit{\Theta}_N (x) \right\vert  \leq \max_{\vartheta
\in \mathit{\Theta}_N (x)} \exp\left(\sum_{y\in {\sf V}_\vartheta}
n(y) \log n(y) \right).
\end{equation}
\end{lemma}
{\it Proof:} Obviously,
\[
\left\vert \mathit{\Theta}_N (x) \right\vert \leq  \sum_{y: \ y\sim
x} \left\vert \mathit{\Theta}_{N-1} (y) \right\vert \leq \sup_{y: \
y\sim x}n(x) \left\vert \mathit{\Theta}_{N-1} (y) \right\vert ,
\]
which by the induction in $N$ yields
\begin{eqnarray*}
\left\vert \mathit{\Theta}_N (x) \right\vert & \leq &
\max_{\vartheta \in \mathit{\Theta}_N (x)} n(x) n(x_1) \cdots n(x_{N-1}) \\[.2cm]
 & = & \max_{\vartheta \in \mathit{\Theta}_N (x)}\exp\left( \sum_{y\in {\sf V}_\vartheta}
 \nu_\vartheta (y) \log n(y) \right) \nonumber \\[.2cm] & \leq & \max_{\vartheta \in \mathit{\Theta}_N (x)}
  \exp\left(\sum_{y\in {\sf V}_\vartheta} n(y) \log n(y) \right), \nonumber
\end{eqnarray*}
where we have used claim (ii) of Lemma \ref{pathpn}.
 $\square$ \vskip.1cm \noindent
Similarly, one proves that
\begin{equation}
  \label{13}
\left\vert \mathit{\Sigma}_N (x) \right\vert \leq \max_{\vartheta
\in \mathit{\Sigma}_N (x)} \exp\left( \sum_{y\in {\sf V}_\vartheta}
\log n(y)\right).
\end{equation}

\subsection{Graphs with tempered growth of vertex degree}

For an increasing function $g: \mathbb{N} \rightarrow (0, +\infty)$
 and an animal ${\sf A}$, we set
\begin{equation*}
G({\sf A}; g) = \frac{1}{|{\sf V}({\sf A})|} \sum_{x \in
{\sf V}({\sf A})} g\left(n(x) \right).
\end{equation*}
If $n_{\sf G}< \infty$, see (\ref{N}), then $G({\sf A}; g) \leq
g(n_{\sf G})$ for any animal and any function $g$. We say that the
vertex degree in ${\sf G}$ is of {\it tempered} growth if
\begin{equation}
  \label{110}
  \max_{{\sf A} \in
\mathcal{A}_{N}(x)}G({\sf A};g) \leq \gamma.
\end{equation}
More precisely, we mean the following.
\begin{definition}
  \label{goodadf}
The graph ${\sf G}$ is said to be $g$-tempered (resp. strongly
$g$-tempered) if there exists a number $\gamma >0$ such that, for
every $x \in {\sf V}$, there exists a strictly increasing sequence
$\{N_k\}_{k\in \mathbb{N}} \subset \mathbb{N}$ (resp. there exists
$N_x \in \mathbb{N}$) such (\ref{110})  holds for all $N=N_k$, $k\in
\mathbb{N}$ (resp. for all $N\geq N_x$).
\end{definition}
\begin{lemma}
  \label{gpn2}
For $g(t) = t \log t$, $t\in \mathbb{N}$, let ${\sf G}$ be
$g$-tempered. Then there exists $q_{\sf G}
>1$ such that, for any $x\in {\sf V}$, there exists a
strictly increasing sequence $\{N_k\}_{k\in \mathbb{N}}\subset
\mathbb{N}$ such that the estimate
\begin{equation}
  \label{g3}
\left\vert \mathcal{A}_{N} (x) \right\vert \leq q_{\sf G}^{N}
\end{equation}
holds for all $N= N_k$, $k\in \mathbb{N}$. If ${\sf G}$ is strongly
$g$-tempered, then for any $x\in {\sf V}$, there exists $N_x \in
\mathbb{N}$ such that (\ref{g3}) holds for all $N\geq
N_x$.
\end{lemma}
{\it Proof:} Let ${\sf A}$ be an animal such that $x\in {\sf V}({\sf
A})$. Consider the multi-graph $\widetilde{\sf A}$ which has the
same vertices as ${\sf A}$ but doubled edges. This means that the
edge set of $\widetilde{\sf A}$ consists of the pairs $e,
\tilde{e}$, both connecting the same vertices, such that $e \in {\sf
E}({\sf A})$. Then the graph $\widetilde{\sf A}$ is Eulerian, see,
e.g., \cite[page 51]{BM}, and hence there exists a path in
$\widetilde{\sf A}$, which originates and terminates at $x$, enters
each $y\in {\sf V}({\sf A})$, and traverses each edge of
$\widetilde{\sf A}$ exactly once. Therefore, ${\sf A} = {\sf
G}_\vartheta$ for some path $\vartheta (x,x)\in \mathit{\Theta}_M
(x)$ with $M= 2 |{\sf E} ({\sf A})|$. Thus, by (\ref{12}) we have
\begin{eqnarray*}
|\mathcal{A}_{N} (x)| \leq  |\mathit{\Theta}_M(x)|& \leq &
\max_{\vartheta \in \mathit{\Theta}_M (x)} \exp\left( \sum_{y\in
{\sf
V}_\theta} g (n(y)) \right) \\[.2cm] & = & \max_{{\sf A} \in
\mathcal{A}_{N} (x)} \exp\left( \sum_{y\in {\sf
V}({\sf A})} g (n(y)) \right) \nonumber \\[.2cm] & \leq & \max_{{\sf A} \in
\mathcal{A}_{N} (x)} \exp\left( N G({\sf A}; g) \right).\nonumber
\end{eqnarray*}
Then we set $q_{\sf G} = e^\gamma$ and obtain (\ref{g3}) from
(\ref{110}).
  $\square$ \vskip.1cm \noindent
In the same way, by means of (\ref{13}) one proves the following
\begin{lemma}
  \label{gepn}
For $g(t) = \log t$, $t\in \mathbb{N}$, let ${\sf G}$ be
$g$-tempered. Then there exists $\bar{q}_{\sf G}
>1$ such that, for any $x\in {\sf V}$, there exists a
strictly increasing sequence $\{N_k\}_{k\in \mathbb{N}}\subset
\mathbb{N}$ such that the estimate
\begin{equation}
  \label{g5}
\left\vert \mathit{\Sigma}_{N} (x) \right\vert \leq \bar{q}_{\sf
G}^{N}
\end{equation}
holds for all $N= N_k$, $k\in \mathbb{N}$. If ${\sf G}$ is strongly
$g$-tempered, then for any $x\in {\sf V}$, there exists
$\overline{N}_x \in \mathbb{N}$ such that the estimate (\ref{g5})
holds for all $N\geq \overline{N}_x$.
\end{lemma}

\subsection{Capacity of animals}

For an animal, ${\sf A} \subset {\sf G}$, by $\rho_{\sf A}(x,y) $, we denote the length of the shortest path
$\vartheta (x,y)$ in ${\sf A}$, i.e., such that ${\sf G}_\vartheta \subset {\sf A}$.
We shall be interested in estimating the
number of vertices in subsets ${\sf B} \subset {\sf V}({\sf A})$,
which have the following property.
\begin{definition}
  \label{an1df}
Given $\lambda >1$, a set, ${\sf B}\subset {\sf V}({\sf A})$, is
said to be $\lambda$-admissible in ${\sf A}$ if $\rho_{\sf A}(x, y)
\geq \lambda$ for any distinct $x, y \in {\sf B}$. The quantity
\begin{equation*}
C({\sf A};\lambda) =  \max\{|{\sf B}|: {\sf B} \   {\rm is} \ \lambda-{\rm admissible}  \ {\rm in}  \ {\sf A}\}
\end{equation*}
is called the $\lambda$-capacity of ${\sf A}$.
\end{definition}
Hence, if ${\sf A}' \subset {\sf A}$ is a connected spanning subgraph, then
\begin{equation}
 \label{span}
C({\sf A};\lambda) \leq C({\sf A}';\lambda).
\end{equation}
If $\vartheta$ is a simple path of length $N$, then
\begin{equation}
  \label{14}
C({\sf G}_\vartheta; \lambda) \leq
 1 + N/\lambda.
 \end{equation}

\begin{lemma}
  \label{anpn}
Let ${\sf A}$ be an animal of size $N$. Then, for any $\lambda>0$,
\begin{equation}
  \label{15}
C({\sf A};\lambda)  \leq \max\left\{ 1; 2N/\lambda\right\}.
\end{equation}
\end{lemma}
{\it Proof:} Suppose first that $N\leq \lambda$. As any simple path
in ${\sf A}$ cannot be longer than $N-1$, one has $\rho (x,y) \leq
N-1$ for any $x, y \in {\sf V}({\sf A})$. Thus, any
$\lambda$-admissible set can contain at most one element, and hence
(\ref{15}) holds. For $N>\lambda$, we use the induction in $N$
assuming that (\ref{15}) holds for all $N' < N$.

Let ${\sf T} \subset {\sf A}$ be a spanning tree for ${\sf A}$. We
are going to estimate its capacity and then to use (\ref{span}). Let
${\sf B}\subset {\sf V}({\sf T})$ be such that
\begin{equation*}
 \min_{x, y \in {\sf B}, \ x\neq y} \rho_{\sf T} (x, y) \geq \lambda.
\end{equation*}
To estimate $|{\sf B}|$, we pick $z, z'\in {\sf V}({\sf A})$ such
that the simple path $\vartheta (z,z')$ in ${\sf T}$ is the longest
one among such paths (backbone); that is, $\rho_{\sf T} (z, z')$ is the
diameter of ${\sf T}$, see Fig. \ref{DrzewkaT}. Let ${\sf T}_0 \subset {\sf T}$ be the graph
generated by this path. Then we split
\begin{equation*}
 {\sf E}({\sf T}) = {\sf E}({\sf T}_0) \cup {\sf E}' \cup {\sf E}'',
\end{equation*}
where ${\sf E}'$ (resp. ${\sf E}''$) consists of those edges of
${\sf T}$ which have exactly one endpoint (resp. no endpoints) in
${\sf V}({\sf T}_0)$. Note that ${\sf E}'$ and ${\sf E}''$ may be
void. Then the graph $({\sf V}({\sf T}), {\sf E}({\sf T}_0) \cup
{\sf E}'')$ is disconnected and falls into $r+1$ connected
components ${\sf T}_0, {\sf T}_1 , \dots , {\sf T}_r$, $r\geq 0$. As
${\sf T}$ is a tree, $r = |{\sf E}'|$; that is, ${\sf E}' = \{
\langle z_1 , z'_1 \rangle, \dots \langle z_r , z'_r \rangle\}$ with
$z_s \in {\sf V}({\sf T}_0)$, $s=1, \dots , r$. We call $z_s$ the
{\it root} of ${\sf T}_s$ in ${\sf T}_0$. As $\rho_{\sf T} (z, z')$
is the diameter of ${\sf T}$, both $z$ and $z'$ cannot be among the
roots. Note also that some of the trees ${\sf T}_s$ may have common
roots.

\begin{figure}[h]
\centerline{\includegraphics[width=0.6\textwidth]{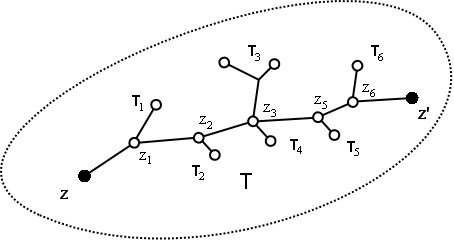}}
\caption{Trees ${\sf T}, {\sf T}_0,{\sf T}_1, \ldots, {\sf
T}_6$.}\label{DrzewkaT}
\end{figure}

Set $N_s = |{\sf V}({\sf T}_s)|$, $s= 0, 1 , \dots , r$. As
$\rho_{\sf T} (z, z')$ is the diameter of ${\sf T}$, we have $N_s
\leq N_0$ for all $s= 1 , \dots , r$. If $N_0 < \lambda$, the
diameter of ${\sf A}$ is less than $\lambda$ and hence $|{\sf
B}|\leq 1$, which yields (\ref{15}). Thus, we assume in the sequel
that $N_0 \geq \lambda$. If $r=0$, then $N_0= N$ and ${\sf B}\subset
{\sf V}({\sf T}_0)$. In this case, by (\ref{14}) we have
\begin{equation*}
  |{\sf B}| \leq 1 + (N-1)/\lambda < 2 N /\lambda.
\end{equation*}
For $N_0 < N$, we number the trees $T_s$ in such a way that, for
some $k \in \{0, 1,  \dots r\}$,  $N_s \geq \lambda /2$ for $s= 0 ,
1 , \dots , k$, and $N_s < \lambda /2$ for $s=k+1, \dots , r$. Then
by the inductive assumption, we have
\begin{equation}
  \label{19}
  \left\vert
  {\sf B} \cap {\sf V}({\sf T}_s) \right\vert \leq 2 N_s /\lambda,
   \quad s= 0, \dots , k.
\end{equation}
Since $N = \sum_{s=0}^r N_s$, for $k=r$, we have from (\ref{19})
that $|{\sf B}| \leq 2N/\lambda$ and hence (\ref{15}) holds. Suppose
now that $k< r$. Then for any $y\in {\sf V}({\sf T}_s)$ with $s= k+1
, \dots , r$, we have that $\rho_{\sf T} (y , z_s) \leq N_s
<\lambda/2$. Therefore, each such a tree can contain at most one
element of ${\sf B}$, and the trees with the same root can contain
at most one such element in common. If none of them contain elements
of ${\sf B}$, we have
\[ \left\vert  {\sf B}
\right\vert = \sum_{s=0}^k  \left\vert  {\sf B} \cap {\sf V}({\sf
T}_s) \right\vert \leq \frac{2}{\lambda} \sum_{s=0}^k  N_s \leq 2N
/\lambda,
\]
which again yields (\ref{15}). Let us forget about those trees which
do not contain elements of ${\sf B}$ and suppose that, for some $n
\in \{k+1 , \dots , r\}$, each of ${\sf T}_{k+1}, \dots ,{\sf T}_n$
contains a single element  $\tilde{y}_s \in {\sf B}$, $s = k+1 ,
\dots , n$. Then all  their roots $z_{k+1} , \dots , z_n$ are
distinct, and, for all such $s$,
\begin{equation}
  \label{y}
\rho_{\sf T} (\tilde{y}_s , z_s) \leq N_s .
\end{equation}
The total number of $\tilde{y}_s$'s is $n-k$. Let us now estimate
the maximum possible number of elements of ${\sf B}$ in the tree
${\sf T}_0$. To this end we consider
\begin{equation}
  \label{20}
{\sf D}_s = \{ y \in {\sf V}({\sf T}_0) : \rho_{\sf T} (y, z_s) <
\lambda - N_s  \}, \quad  s = k+1 , \dots, n.
\end{equation}
Each ${\sf D}_s$ is in  fact an interval of the path $\vartheta(z,
z')$, centered at $z_s$. As for any $y \in {\sf D}_s$, we have
$\rho_{\sf T} (y, \tilde{y}_s) < \lambda$; hence, none of ${\sf
D}_s$ can contain elements of ${\sf B}$.  Some of ${\sf D}_s$ can
overlap, which reduces the part of ${\sf T}_0$ free of elements of
${\sf B}$. If ${\sf D}_s$ and ${\sf D}_{s'}$ overlap, then
\begin{equation}
  \label{21}
\lambda \leq \rho_{\sf T} (\tilde{y}_{s},\tilde{y}_{s'}) \leq  N_{s}
+ N_{s'}  + \rho_{\sf T} (z_{s},z_{s'})
\end{equation}
which gives the lower bound for $\rho_{\sf T} (z_{s},z_{s'})$.

Suppose now that the roots $z_{k+1} , \dots , z_n$, and hence the
corresponding intervals (\ref{20}), are distributed among $q$
groups, $q \geq 1$, consisting of $l_1, \dots , l_q$ elements, $l_1
+ \cdots + l_q = n-k$. We also suppose that consecutive intervals in
each group overlap (if the corresponding $l_j>1$), whereas the
intervals belonging to distinct groups do not overlap. The roots are
numbered in such a way that, for $j=0, 1, \dots , q-1$, the $j$-th
group is
\[
Z_j =  \{ z_{t_j + 1} , \dots , z_{t_{j + 1}}\}, \quad t_j =  k+l_1 + \cdots + l_j, \ \
  \ \ l_0 =0.
\]
For such a group, let $y^*_j$ (resp. $z^*_{j+1}$) be the closest to
$z_{t_{j }+ 1}$ (resp. to $z_{t_{j + 1}}$) element of ${\sf B}\cap
{\sf V}({\sf T}_0)$. Then $\rho_{\sf T} (y^*_j,z_{t_{j }+ 1}) +
\rho_{\sf T} (\tilde{y}_{t_{j }+ 1},z_{t_{j }+ 1}) \geq \lambda$ and
$\rho_{\sf T} (z^*_{j+1},z_{t_{j + 1}}) + \rho_{\sf T}
(\tilde{y}_{t_{j +1}+ 1},z_{t_{j +1}}) \geq \lambda$, which yields,
see (\ref{y}),
\begin{equation*}
\rho_{\sf T} (y^*_j,z_{t_{j }+ 1}) \geq \lambda - N_{t_j +1}, \quad
\rho_{\sf T} (z^*_{j+1},z_{t_{j + 1}}) \geq \lambda - N_{t_{j +1}}.
\end{equation*}
In what follows, the elements of ${\sf B}\cap {\sf V}({\sf T}_0)$
are contained in the paths $\vartheta (z, y^*_0)$, $\vartheta (
y^*_j, z^*_j)$, $j= 1 , \dots , q-1$, and $\vartheta (z^*_q, z')$.
The number of elements of ${\sf B}$ in each such a path can be
estimated by (\ref{14}); thus, we have to estimate the total length
of such paths. The latter quantity is equal to the length of
$\vartheta ( z,z')$ minus the total length of the intervals
(\ref{20}); that is,
\begin{eqnarray} \label{23y}
L & \stackrel{\rm def}{=} &  |\vartheta (z, y^*_0)| + |\vartheta (z^*_q, z')|
 + \sum_{j=1}^{q-1}|\vartheta ( y^*_j, z^*_j)| \\[.1cm] & \leq & N_0-1 -
  \sum_{j=0}^{q-1}|\vartheta ( z_{t_j +1}, z_{t_{j+1}})|. \nonumber
\end{eqnarray}
The latter summand can be estimated by means of (\ref{21}), which,
for $j= 0, 1 , \dots , q-1$, yields
\begin{equation*}
|\vartheta ( z_{t_j +1}, z_{t_{j+1}})| \geq \lambda (l_{j+1} - 1) -
\sum_{s=t_j + 1}^{t_{j+1} -1} (N_s + N_{s+1}).
\end{equation*}
Applying this estimate in (\ref{23y}) and taking into account that the
total number of elements of ${\sf B}$ in ${\sf T}_1 , \dots , {\sf
T}_k$ was estimated in (\ref{19}), we arrive at
\begin{eqnarray*}
|{\sf B} | & \leq & q + 1 + L/\lambda + (n-k) + \frac{2}{\lambda}\sum_{s=1}^k N_s \\[.2cm]
& \leq & q + 1 + (n-k) + \frac{2}{\lambda}\sum_{s=1}^k N_s + (N_0 -1)/\lambda - (n-k) - q \\[.2cm]
& - & \frac{1}{\lambda} \sum_{j=0}^{q-1} (N_{t_j+1} + N_{t_{j+1}}) + \frac{2}{\lambda}\sum_{s=k+1}^n N_s \\[.2cm]
& = & \frac{2}{\lambda}\sum_{s=0}^n N_s + 1 -
\frac{1}{\lambda}\left( N_0+ \sum_{j=0}^{q-1} (N_{t_j+1} + N_{t_{j+1}})\right)\\[.2cm]
 & \leq & \frac{2}{\lambda}\sum_{s=0}^n N_s \leq 2 N /\lambda,
\end{eqnarray*}
where we have taken into account that $N_0 \geq \lambda$ and $N =
N_0 + \cdots + N_n + \cdots + N_r$. $\square$ \vskip.1cm It is
worthwhile to note that the estimate in (\ref{15}) is optimal, that
is, for each $\varepsilon>0$, one can pick ${\sf A}$ of size $N$ and
$\lambda>0$ such that $C({\sf A}, \lambda) > 2N/\lambda -
\varepsilon$. An instance can be ${\sf G}_\vartheta$, $\vartheta$
being a simple path, cf. (\ref{14}). In this case, $$C({\sf
G}_\vartheta, |\vartheta|) =2 > \frac{2(|\vartheta| +
1)}{|\vartheta|} - \varepsilon, $$ for sufficiently big
$|\vartheta|$.

\subsection{Balls in repulsive graphs}
\label{repgr}

We recall that ${\sf B}_N (x)=\{y\in {\sf V}: \rho(x,y) \leq N\}$,
$N\in \mathbb{N}$, denotes the ball in ${\sf G} = ({\sf V}, {\sf
E})$ of radius $N$ centered at $x$, cf. (\ref{U2}) and Corollary
\ref{1co}. Further properties of such sets are described in the
following statement.
\begin{lemma}
 \label{Qpn}
Let ${\sf G}$ be in $\mathbb{G}_{-} (\phi)$. Then for each $x\in
{\sf V}$, there exists a strictly increasing sequence $\{N_k\}_{k\in
\mathbb{N}} \subset \mathbb{N}$, such that, for all $k\in \mathbb{N}$,
\begin{equation}
 \label{10qq}
\max_{y\in {\sf B}_{N_k} ( x) } n(y) \leq \phi^{-1} (2 N_k +1).
\end{equation}
If ${\sf G}\in\mathbb{G}_{+} (\phi)$, then, for every $x\in {\sf V}$, there
exists $N_x \in \mathbb{N}$ such that the estimate
\begin{equation}
 \label{10qu}
\max_{y\in {\sf B}_N (x) } n(y) \leq \phi^{-1} (2 N)
\end{equation}
 holds for all $N\geq N_x$.
\end{lemma}
{\it Proof:} First we consider the case of ${\sf G}\in
\mathbb{G}_{-}(\phi)$. Let $x_1 \in{\sf V}_*^c $ be the closest
vertex to $x$ such that $n(x_1)  > n({x})$. If there are several
such vertices, we take the one with the biggest degree. In the same
way, we  pick $x_2$, being the closest vertex to $x$ such that
$n(x_2)
> n(x_1)$. Then we set $N_1 = \rho(x, x_2) - 1$, which yields
$n(x_1) = \max_{y\in {\sf B}_{N_1}( x) } n(y)$. By (\ref{7}) $\rho
(x_1 , x_2) \geq \phi(n(x_1))$; hence, $2N_1 + 1 \geq  \rho (x ,
x_2) + \rho (x , x_1) \geq \phi(n(x_1))$. Thus, (\ref{10qq}) holds
for $N=N_1$. Next we take $x_3$ such that $n(x_3)
> n(x_2)$ and set $N_2 = \rho (x, x_3) - 1$. In this way, we
construct the whole sequence $\{N_k\}_{k\in \mathbb{N}}$ for which
(\ref{10qq}) holds.

Let now ${\sf G}$ be in $\mathbb{G}_{+}(\phi)$. Given $x$, let
$\tilde{x}\in {\sf V}_*^c$ be the closest vertex  to $x$, see
(\ref{NU}). If there are several such vertices, we take the one with
the biggest degree. Consider the following cases: (i) $\rho(x ,
\tilde{x})
> \phi (n(\tilde{x}))/2$; (ii) $\rho(x , \tilde{x}) \leq \phi
(n(\tilde{x}))/2$. The latter one includes the case $\tilde{x} = x$,
i.e., $x$ itself is  in ${\sf V}^c_*$. In case (i), we set $N_x$ to
be the smallest integer number such that $N_x > \phi(n_*) / 2$. Then
for $N\geq N_x$, we have the following possibilities: (a) $N <
\rho(x , \tilde{x})$; (b) $N \geq \rho(x , \tilde{x})$. If (a)
holds, then the ball ${\sf B}_{N} ( x)$ contains only elements of
${\sf V}_*$ and hence (\ref{10qu}) holds true. In case (b), we have
: (c) $\max_{y\in {\sf B}_{N} (x) } n(y) = n(\tilde{x})$; (d) there
exists $z\neq \tilde{x}$ such that $\max_{y\in {\sf B}_N (x) } n(y)
= n(z)$. If (c) holds, we again obtain (\ref{10qu})  since $N \geq
\rho(x, \tilde{x})> \phi (n(\tilde{x}))/2$, by (i). If (d) holds, by
(\ref{7}) we have $\rho(z, \tilde{x}) \geq \phi( n(z))$. On the
other hand, by the triangle inequality $\rho(z, \tilde{x}) \leq
\rho(x, \tilde{x}) + \rho(z, {x}) \leq 2N$, which again yields
(\ref{10qu}). If (ii) holds, let $x_1$ be the closest vertex to $x$
such that $n(x_1)
> n(\tilde{x})$, again we take  the one with the maximum degree
among such vertices. Then we set $N_x = \rho (x, x_1)$. For $N\geq
N_x$, let $z$ be such that $n(z) = \max_{y \in {\sf B}_N (x) }
n(y)$. Then by (\ref{7}) $\rho (\tilde{x}, z) \geq \phi (n(z))$. By
the triangle inequality this yields
\[
N \geq \rho (x, z) \geq \rho (\tilde{x}, z) - \rho (x , \tilde{x})
\geq \phi (n(z)) - \phi (n(\tilde{x})) /2 \geq \phi (n(z))/2,
\]
which yields (\ref{10qu}) and hence completes the proof. $\square$


\section{The proof of Theorems 2 and 3 and Propositions 7 and 9}
\label{4S}

The proof of the statements in question relies upon showing that the graphs ${\sf
G}\in \mathbb{G}_{-}(\phi)$ (resp. ${\sf G}\in
\mathbb{G}_{+}(\phi)$) are $g$-tempered (resp. strongly
$g$-tempered) if $g$ and $\phi$ satisfy a certain condition. Then we
apply Lemmas \ref{gpn2} and \ref{gepn} and obtain the result we want. To realize this we introduce one more notion. Set
\[
n_{\sf A} = \max_{x\in {\sf V}( {\sf
A})} n(x).
\]
\begin{definition}
 \label{gbdf}
Given ${\sf G}\in \mathbb{G}_{\pm} (\phi)$, an ${\sf A}\subset {\sf G}$ is said to be a good
animal if
\begin{equation}
 \label{14z}
\left\vert {\sf V}( {\sf A})\right\vert \geq \phi(n_{\sf A})/2.
\end{equation}
\end{definition}
By $\mathcal{A}_{\rm good}$ we denote the set of all good animals,
cf. (\ref{10qq}) and (\ref{10qu}).
\begin{lemma}
 \label{goodpn1}
Let the functions $g$ and $\phi$ be such that the following holds
\begin{equation}
 \label{13z}
\sum_{k=1}^\infty \frac{g(t_{k+1})}{\phi(t_k)} < \infty,
\end{equation}
for some strictly increasing sequence $\{t_k\}_{k\in \mathbb{N}}
\subset \mathbb{N}$. Then any  ${\sf G}\in \mathbb{G}_{-} (\phi)$
(resp. any ${\sf G}\in \mathbb{G}_{+} (\phi)$) is  $g$-tempered
(resp. is strongly $g$-tempered).
\end{lemma}
{\it Proof:} First we consider the case ${\sf G}\in \mathbb{G}_{-}
(\phi)$. Let the sequence in (\ref{13z}) be such that $t_1 = n_*$.
The proof will be done by showing that: (a) the upper bound as in
(\ref{110}) holds for any good animal; (b) for each $x\in {\sf V}$,
one can pick $\{N_k\}_{k\in \mathbb{N}}$ such that each ${\sf A} \in
\mathcal{A}_{N_k} (x)$, $k\in \mathbb{N}$, is a good animal.

For ${\sf A} \in \mathcal{A}_{\rm good}$, we set
\begin{eqnarray}
  \label{g1}
{\sf M}_k ({\sf A}) & = & \{ x\in {\sf V} ({\sf A}) : n(x) \in (t_k,
t_{k+1}] \}, \quad k = 1 , \dots , l, \quad \\[.2cm] m_k ({\sf A})
& = & \left\vert {\sf M}_k ({\sf A}) \right\vert, \nonumber
\end{eqnarray}
where $l \in \mathbb{N}$ is the smallest number for which $n_{\sf A}
\leq t_{l +1}$, see (\ref{14z}). By (\ref{7}) we then get
$\rho(x,y) \geq \phi(t_k)$ for each $x,y \in {\sf M}_k ({\sf A})$.
Hence, by Lemma \ref{anpn} we have
\[
 m_k ({\sf A}) \leq C({\sf A}, \phi(t_k) ) \leq 2 |{\sf V}({\sf A})|
 /  \phi (t_k),
\]
which leads to the following estimate, cf. (\ref{110}),
\begin{equation}
  \label{g2}
G({\sf A};g) \leq \frac{1}{ |{\sf V}({\sf A})|} \sum_{k=1}^{l}
g(t_{k+1}) m_k ({\sf A}) \leq 2 \sum_{k=1}^\infty
\frac{g(t_{k+1})}{\phi(t_k)} \ \stackrel{\rm def}{=} \ \gamma (g,
\phi).
\end{equation}
Let $x$ be an arbitrary vertex.  For this  $x$, let $\{N_k\}_{k\in
\mathbb{N}}$ be the sequence  as in Lemma \ref{Qpn}. Then, for
any ${\sf A}$ such that $x\in {\sf V}({\sf A})$ and $|{\sf V}({\sf
A})| = N_1$, we have ${\sf V}({\sf A})\subset {\sf B}_{ N_1 -1} ( x)$.
Then  by (\ref{10qq})
\[
2 N_1 > 1 + 2(N_1 - 1) \geq \phi\left(\max_{y \in {\sf V}({\sf A})}
n(y)  \right),
\]
which yields ${\sf A} \in \mathcal{A}_{\rm good}$. Hence,
(\ref{g2}) holds for any ${\sf A} \in \mathcal{A}_{N_1} (x)$. Then
we repeat the same procedure with $N_2$, $N_3$, and so on.  For ${\sf G}\in \mathbb{G}_{+} (\phi)$, the
proof  follows along the same
line of arguments, with the only difference that by (\ref{10qu}) we show that
$\mathcal{A}_{N}(x)\subset \mathcal{A}_{\rm good}$ whenever $N\geq N_x$.
$\square$ \vskip.1cm
The proof of Theorem \ref{1tm} readily follows from Lemmas \ref{gpn2} and \ref{goodpn1} with $g(t) = t \log t$.
In the same way, by taking $g(t) = \log t$ (resp. $g(t) = t^{\theta+1}$, cf. (\ref{warR})) we prove Theorem \ref{2tm} (resp. Proposition \ref{RIpn}), see Lemma \ref{gepn}.

To prove Proposition \ref{5apn} we proceed as follows. Set $$S({\sf A}) = \sum_{x\in {\sf V}({\sf A})}Y_x, $$ cf. (\ref{SNx}).
Then, for $Y>0$ and $t>0$, we have, cf. (\ref{Ex}),
\begin{eqnarray*}
\mathbb{P}\bigg{(}S({\sf A}) \geq Y |{\sf V}({\sf A})|\bigg{)} & \leq & \exp( - t Y |{\sf V}({\sf A})|) \mathbb{E}\exp\left( t \sum_{x\in {\sf V}({\sf A})} Y_x \right)\\[.2cm] & = & \exp\left(- t \sum_{x\in {\sf V}({\sf A})}(Y - w_x(t)/t) \right)\\[.2cm]
& \leq & \exp\left(- t Y |{\sf V}({\sf A})| +  t C \sum_{x\in {\sf V}({\sf A})} n(x) \log n(x) \right),
\end{eqnarray*}
which holds for small enough $t>0$, see (\ref{Ex1}) and (\ref{Ex2}). For $\phi$ satisfying (\ref{U4}), the graph in question is $g$-tempered with
$g(t)= t \log t$, see Lemma \ref{goodpn1}. Given $x\in {\sf V}$, let $\{N_k\}_{k\in \mathbb{N}}$ be the sequence as in Definition \ref{goodadf}. Then, for ${\sf A} \in \mathcal{A}_{N_k} (x)$, by (\ref{110}) and the latter estimate we obtain
\begin{eqnarray*}
\mathbb{P}\bigg{(}S_{N_k}(x) \geq Y N_k \bigg{)}  \leq \exp\left( - t N_k (Y - \gamma C)\right) .
\end{eqnarray*}
Now we take $Y > \gamma C$ and obtain (\ref{ex3}) by applying the Borel-Cantelli lemma.

\section*{Acknowledgment}
This work was supported in part by the DFG through SFB 701: ``Spektrale Strukturen
und Topologische Methoden in der Mathematik"  and through the
research project 436 POL 125/113/0-1, which is cordially acknowledged by the authors.





\bibliographystyle{model1b-num-names}



\end{document}